\documentclass[12pt,reqno]{amsart}
\begin{document}

\baselineskip=15.2pt

\numberwithin{equation}{section}

\newtheorem{theorem}{Theorem}[section]
\newtheorem{lemma}[theorem]{Lemma}
\newtheorem{proposition}[theorem]{Proposition}
\newtheorem{corollary}[theorem]{Corollary}

\theoremstyle{definition}
\newtheorem{remark}[theorem]{Remark}

\def\quot{/\!\!/}

\title[Character varieties of virtually nilpotent K\"ahler groups]{Character varieties
of virtually nilpotent K\"ahler groups and $G$--Higgs bundles}

\author[I. Biswas]{Indranil Biswas}

\address{School of Mathematics, Tata Institute of Fundamental
Research, Homi Bhabha Road, Bombay 400005, India}

\email{indranil@math.tifr.res.in}

\author[C. Florentino]{Carlos Florentino}

\address{Departamento Matem\'atica, Instituto Superior T\'ecnico, Av. Rovisco
Pais, 1049-001 Lisbon, Portugal}

\email{cfloren@math.ist.utl.pt}

\subjclass[2000]{20G20, 14J60}

\keywords{K\"ahler group, character variety, $G$--Higgs bundle, virtually
nilpotent group}

\date{}

\begin{abstract}
Let $G$ be a connected complex reductive affine algebraic group, and let $K\,\subset\, G$
be a maximal compact subgroup. Let $X$
be a compact connected K\"ahler manifold whose fundamental group $\Gamma$ is virtually
nilpotent. We prove that the character variety $\text{Hom}(\Gamma,\, G)/\!\!/G$
admits a natural strong deformation retraction to the subset $\text{Hom}(\Gamma,\, K)/K\,\subset\,
\text{Hom}(\Gamma,\, G)/\!\!/G$. The natural action of ${\mathbb C}^*$ on the
moduli space of $G$--Higgs bundles over $X$ extends to an action of $\mathbb C$. This
produces the above mentioned deformation retraction.
\end{abstract}
\maketitle

\section{Introduction}

Let $G$ be a complex reductive affine algebraic group, and let $\Gamma$ be a
finitely presentable group. Let $\mathcal{R}_{\Gamma}(G)\,:=\,{\rm Hom}(\Gamma,G)\quot G$
be the geometric invariant theoretic (GIT) quotient, of the space of all homomorphisms from
$\Gamma$ to $G$, for the conjugation action of $G$; it is known as the $G$--character variety of
$\Gamma$. These moduli spaces $\mathcal{R}_{\Gamma}(G)$ play
important roles in hyperbolic geometry \cite{CS}, the theory of bundles and
connections \cite{Si}, knot theory and quantum field theories \cite{Gu} (see also the
references in these papers).

Some particularly relevant cases of $\Gamma$ include, for instance, the fundamental
group of a compact connected K\"ahler manifold. These are called K\"ahler groups.
If $\Gamma$ is the fundamental group of a compact connected K\"ahler manifold $X$,
then the corresponding
character variety $\mathcal{R}_{\Gamma}(G)$ can be identified with
a certain moduli space of $G$--Higgs bundles over $X$ \cite{H,Si,BG}; this identification
is continuous but not holomorphic. Let $K$ be a maximal compact subgroup of $G$. The
above identification between $\mathcal{R}_{\pi_1(X,x_0)}(G)\,=\,
\mathcal{R}_{\Gamma}(G)$ and a moduli space of $G$--Higgs
bundles on $X$ is an extension of the identification between ${\rm Hom}(\Gamma,K)/ K$
and the moduli space of semistable principal $G$--bundles on $X$ with vanishing characteristic
classes of positive degrees \cite{RS}, \cite{Do}, \cite{UY}, \cite{NS}.

In investigations of the
topology of $\mathcal{R}_{\Gamma}(G)$ there are some notable situations
where the analogous orbit space $\mathcal{R}_{\Gamma}(K)\,:=\,{\rm Hom}(\Gamma,K)/K$
is a strong deformation retract of $\mathcal{R}_{\Gamma}(G)$.
This happens when $\Gamma$ is a free group \cite{FL} or a free
abelian group \cite{FL3,BF2,PS} or a nilpotent group \cite{B}.
It should be mentioned that such a deformation retraction is not to be expected for
arbitrary finitely presented groups
$\Gamma$, not even for general K\"ahler groups. For example, this fails for
surface groups \cite{BF1}.

In this article, we consider the case where $\Gamma$ is a virtually
nilpotent K\"ahler group. This means that $\Gamma$ is the fundamental group of a compact
connected K\"ahler manifold and it has a finite index
subgroup which is nilpotent. Since any finite group is
the fundamental group of a complex projective manifold \cite[p. 6,
Example 1.11]{ABCKT}, this class of groups include all finite groups.

Our approach uses non-abelian Hodge theory. When applied to a general
compact connected K\"ahler manifold $X$, the non-abelian Hodge theory
provides a natural correspondence between the flat
principal $G$--bundles over $X$ and a certain class of $G$--Higgs bundles
on $X$. If $X$ is a smooth complex projective variety, then this correspondence
sends the flat principal $G$--bundles on $X$ to the semistable $G$--Higgs bundles
on $X$ with vanishing characteristic classes of positive degrees. More generally, if
$X$ is a compact connected K\"ahler manifold, then the same correspondence remains valid
once semistability is replaced by pseudostability.

Let $X$ be a compact connected K\"ahler manifold such that its fundamental group $\Gamma\, :=\,
\pi_1(X,x_0)$ is virtually nilpotent. Take any homomorphism
$\rho\, :\, \Gamma\, \longrightarrow\, G$. Let $(F_G\, ,\theta)$
be the pseudostable $G$--Higgs bundle on $X$ associated to $\rho$.
(If $X$ is a complex projective manifold, then $(F_G\, ,\theta)$
is semistable.) We prove that the underlying principal $G$--bundle $F_G$ is
pseudostable (see Proposition \ref{prop2}); in \cite{FGN}, a similar result is
proved for Higgs $G$--bundles on elliptic curves. In view
of the earlier mentioned identification between $\mathcal{R}_{\Gamma}(G)$ and
a moduli space of $G$--Higgs bundles, this produces a multiplication
action of $\mathbb C$ on $\mathcal{R}_{\Gamma}(G)$ using the multiplication
of Higgs fields by the scalars. The action of $1\,\in\, \mathbb C$ is
the identity map of $\mathcal{R}_{\Gamma}(G)$, and the action of $0$ is a retraction of
$\mathcal{R}_{\Gamma}(G)$ to ${\rm Hom}(\Gamma,K)/K$; this property of the action of $0$
is deduced from the
identification between ${\rm Hom}(\Gamma,K)/K$ and the moduli space of pseudostable
principal $G$--bundles on $X$ with vanishing characteristic classes of positive
degrees \cite{NS}, \cite{Ra}. Also, the action of every element
of $\mathbb C$ fixes the subset ${\rm Hom}(\Gamma,K)/K$ pointwise. Therefore, this
action of $\mathbb C$ on $\mathcal{R}_{\Gamma}(G)$
produces a strong deformation retraction of $\mathcal{R}_{\Gamma}(G)$
to ${\rm Hom}(\Gamma,K)/K$ (see Theorem \ref{thm2}).
 
\section{K\"ahler groups and Flat $G$--bundles}

Let $G$ be a linear algebraic group defined over $\mathbb C$. A Borel subgroup of $G$ is a maximal
Zariski closed connected solvable subgroup of $G$. Any two Borel subgroups of $G$ are
conjugate \cite[p. 134, \S~21.3, Theorem]{Hu}. Let $\Gamma$ be a finitely presentable group.

\subsection{Homomorphisms of virtually nilpotent K\"ahler groups}

Recall that $\Gamma$ is called \emph{virtually solvable} (respectively,
\emph{virtually nilpotent}) if there is a finite index subgroup
\[
\Gamma_{1}\subset\Gamma
\]
such that $\Gamma_{1}$ is a solvable (respectively, nilpotent) group. 

\begin{lemma}
\label{lem:lem1}Let $\Gamma$ be a virtually solvable group. Then,
for any homomorphism $\rho\,:\,\Gamma\,\longrightarrow\, G$, there
is a finite index subgroup \[
\Gamma_{0}\,\subset\,\Gamma\, ,\]
and also a Borel subgroup $B\,\subset\, G$, such that $\rho(\Gamma_{0})\,\subset\, B$.
\end{lemma}

\begin{proof}
Since $\Gamma$ is virtually solvable, there is a solvable subgroup
$\Gamma_{1}\,\subset\,\Gamma$ of finite index. Let $H$ denote the
Zariski closure of the image $\rho(\Gamma_{1})$.
In particular, $H$ is an algebraic subgroup
of $G$. Moreover, $H$ is a solvable subgroup of $G$ because $\Gamma_{1}$
is solvable. Let \[
H_{0}\,\subset\, H\]
 be the connected component of $H$ containing the identity element.
Since $H_{0}$ is a connected solvable subgroup of $G$, there is
a Borel subgroup $B\,\subset\, G$ with \[
H_{0}\,\subset\, B\,.\]

The group $H$ has only finitely many connected components because it is
algebraic. This implies that \[
\Gamma'\,:=\, H_{0}\bigcap\rho(\Gamma_{1})\,\subset\,\rho(\Gamma_{1})\]
is a finite index subgroup of $\rho(\Gamma_{1})$. Indeed, the index
of $\Gamma'$ in $\rho(\Gamma_{1})$ coincides with the number of connected
components of $H$.

Now define \[
\Gamma_{0}\,:=\,\rho^{-1}(\Gamma')\bigcap\Gamma_{1}\,=\,
\rho^{-1}(H_{0})\bigcap\Gamma_{1}\,\subset\,\Gamma_{1}.\]
The index of the subgroup $\Gamma_{0}\,\subset\,\Gamma_{1}$ coincides
with the index of the subgroup $\Gamma'\,\subset\,\rho(\Gamma_{1})$.
In particular, $\Gamma_{0}$ is a subgroup of $\Gamma_{1}$ of finite index.
Since $\Gamma_1$ is a subgroup of $\Gamma$ of finite index,
we now conclude that $\Gamma_{0}$ is a finite index subgroup of $\Gamma$.
We also have $\rho(\Gamma_{0})\,\subset\, H_{0}\,\subset\, B$, so the
proof is complete.
\end{proof}

By a \emph{K\"ahler group} we mean a finitely presentable group isomorphic to the
fundamental group $\pi_{1}(X,x_{0})$ of some compact connected K\"ahler manifold $X$, where
$x_0\in X$ is a base point.

\begin{remark}\label{er1}
Suppose that $\Gamma$ is a virtually solvable K\"ahler group, and $\Gamma_{1}\subset\Gamma$
is a solvable subgroup of finite index. Then $\Gamma_{1}$ is a solvable
K\"ahler group, because finite index subgroups of K\"ahler groups are also K\"ahler groups.
By a recent result of Delzant (see \cite{D}), $\Gamma_{1}$
is virtually nilpotent. So there is a subgroup $\Gamma_{2}\subset\Gamma_{1}$
of finite index such that $\Gamma_{2}$ is nilpotent. Therefore, $\Gamma$
itself is virtually nilpotent.
\end{remark}

In view of Remark \ref{er1}, while considering virtually solvable K\"ahler groups, we can
restrict ourselves to virtually nilpotent K\"ahler groups.

The following proposition is immediate from Lemma \ref{lem:lem1}.

\begin{proposition}\label{pro:virt-nil}
Let $X$ be a compact connected K\"ahler manifold with a virtually nilpotent
fundamental group $\pi_{1}(X,x_{0})$. Let $\rho\,:\,\pi_{1}(X,x_{0})\,\longrightarrow\, G$
be a homomorphism. Then there is a finite index subgroup \[
\Gamma_{0}\,\subset\,\pi_{1}(X,x_{0})\]
and a Borel subgroup $B\,\subset\, G$, such that $\rho(\Gamma_{0})\,\subset\, B$.
\end{proposition}

\subsection{Flat principal $G$--bundles and $G$--Higgs bundles}

Let $G$ be a connected linear algebraic group
defined over $\mathbb{C}$. Let $X$ be a compact connected K\"ahler manifold
equipped with a K\"ahler form $\omega$. There is an equivalence between
the category of pseudostable Higgs $G$--bundles on $X$ with vanishing characteristic classes
of positive degrees and
the category of flat principal $G$--bundles on $X$ \cite[p. 20, Theorem 1.1]{BG}.

Let $p\,:\,Y\,\longrightarrow\, X$ be a finite \'etale
covering with $Y$ connected. So $(Y\, ,p^*\omega)$ is a compact
connected K\"ahler manifold.

\begin{lemma}
\label{lem:pullback}
Let $(E_{G}\, ,\nabla)$ be a flat principal $G$--bundle on $X$, and let
$(F_{G}\, ,\theta)$ be the pseudostable Higgs
$G$--bundle over $X$ associated to $(E_{G}\, ,\nabla)$. Then the pullback
$(p^{*}F_{G}\, ,p^{*}\theta)$ is isomorphic to
the pseudostable Higgs $G$--bundle over $Y$ associated to the flat principal
$G$--bundle $(p^{*}E_{G}\, ,p^{*}\nabla)$.
\end{lemma}

\begin{proof}
First assume that $G$ is reductive. We recall that
a flat $G$--connection on $X$ is called irreducible if it does not admit
a reduction of structure group to a proper parabolic subgroup of $G$.
A flat $G$--connection is called completely reducible if it admits a reduction
of structure group to a Levi subgroup of a parabolic subgroup of $G$ such
that the reduction is irreducible.

Let $(E'\, ,\nabla')$ be a completely reducible flat
principal $G$--bundle on $X$.
Suppose that $$h\,=\, E_K\, \subset\, E'$$
is a harmonic metric on $(E'\, ,\nabla')$. 
Then clearly $p^{*}h\,=\, p^*E_K
\, \subset\, p^{*}E'$
is a harmonic metric on the flat principal $G$--bundle $(p^{*}E'\, ,p^{*}\nabla')$
on $Y$.

On the other hand, if $h_{1}$ is a Hermitian
structure on a polystable Higgs $G$--bundle $(F'\, ,\theta')$ on
$X$ that satisfies the Yang-Mills-Higgs equation, then the pulled
back Hermitian structure $p^{*}h_{1}$ on $(p^{*}F'\, ,p^{*}\theta')$
also satisfies the Yang-Mills-Higgs equation. From this it follows
immediately that the correspondence in \cite[p. 36, Lemma 3.5]{Si}
is compatible with taking finite \'etale coverings.

The correspondence in \cite[p. 20, Theorem 1.1]{BG} for general $G$
is constructed from the correspondence
in \cite[p. 36, Lemma 3.5]{Si}. Therefore, it is also compatible
with taking finite \'etale coverings. In particular, the pseudostable
Higgs $G$--bundle on $Y$ associated to the flat principal $G$--bundle
$(p^{*}E_{G}\, ,p^{*}\nabla)$ coincides with the pullback $(p^{*}F_{G}\, ,p^{*}\theta)$,
where $(F_{G}\, ,\theta)$ as before is the pseudostable Higgs
$G$--bundle over $X$ associated to $(E_{G}\, ,\nabla)$.
\end{proof}

\section{Semistability of holomorphic $G$--bundles underlying Flat $G$--bundles}

{}From now on, $G$ will be assumed to be connected and reductive.

We start, for simplicity, with the projective case. So, let $X$
denote a connected smooth complex projective variety such that $\pi_{1}(X,x_{0})$
is virtually nilpotent.
Note that for any finite index subgroup of $\Gamma_0\,\subset\,\pi_{1}(X,x_{0})$,
the covering of $X$ associated to $\Gamma_0$ is also a connected smooth complex projective
variety.

To define (semi)stability of bundles on $X$, we need to fix a polarization on $X$ (first Chern
class of an ample line bundle), in order to compute the degree and slope of torsionfree
coherent sheaves on $X$. However, for bundles on $X$ with vanishing
characteristic classes of positive degrees, the notion of (semi)stability 
is independent of the choice of polarization. Since we are dealing
with bundles with vanishing characteristic classes of positive degrees, we
will not refer to a particular choice of polarization.

\begin{proposition}
\label{prop1}
Let $X$ be a connected smooth complex projective variety such that
$\pi_{1}(X,x_{0})$ is virtually nilpotent. Let $(F_G\, ,\theta)$ be a
semistable $G$--Higgs bundle on $X$ whose characteristic classes of positive degrees vanish. 
Then the holomorphic principal $G$--bundle $F_G$
is semistable.
\end{proposition}

\begin{proof}
Let $(E_G\, ,\nabla)$ be the flat principal $G$--bundle over $X$ corresponding to the
given semistable $G$--Higgs bundle $(F_G\, ,\theta)$.
Suppose that $(E_{G}\, ,\nabla)$
is given by the homomorphism (its monodromy representation) \begin{equation}
\rho\,:\,\pi_{1}(X,x_{0})\,\longrightarrow\, G\,.\label{e2}\end{equation}
Since $\pi_{1}(X,x_{0})$ is virtually nilpotent, from Proposition
\ref{pro:virt-nil} we know that there is a finite index subgroup
\[
\Gamma_{0}\,\subset\,\pi_{1}(X,x_{0})\]
 and a Borel subgroup $B\,\subset\, G$, such that \begin{equation}
\rho(\Gamma_{0})\,\subset\, B\,.\label{e3}\end{equation}
 Let
\begin{equation}\label{cp}
p\,:\, Y\,\longrightarrow\, X
\end{equation}
be the finite \'etale covering corresponding to the subgroup $\Gamma_{0}$ in
\eqref{e3}. We note that $Y$ is a connected smooth complex projective variety. 

Let $y_{0}\,\in\, p^{-1}(x_{0})\,\subset\, Y$ be the base point of
the covering $Y$. Consider the homomorphism \[
\rho'\,:\,\pi_{1}(Y,y_{0})\,=\,\Gamma_{0}\,\stackrel{\rho\vert_{\Gamma_{0}}}{\longrightarrow}\, B\]
(see \eqref{e3}). Let $(E_{B}\, ,\nabla^{B})$ be the flat principal
$B$--bundle on $Y$ associated to $\rho'$. Let $(F_{B}\, ,\theta_{B})$
be the semistable $B$--Higgs on $Y$ corresponding to
$(E_{B}\, ,\nabla^{B})$. It should be clarified that from
\cite[p. 26, Proposition 2.4]{BG} we know that a Higgs principal bundle on $X$ with
vanishing characteristic classes of positive degrees is semistable if and only
if it is pseudostable. 

Note that $(p^{*}E_{G}\, ,p^{*}\nabla)$ is identified with the flat
principal $G$--bundle on $Y$ obtained by extending the structure group of
the flat principal $B$--bundle $(E_{B}\, ,\nabla^{B})$ using the
inclusion of $B$ in $G$. The correspondence in \cite[p. 20, Theorem 1.1]{BG}
is compatible with extensions of structure group. Therefore, using Proposition
\ref{lem:pullback} we know that the pullback $(p^{*}F_{G}\, ,p^{*}\theta)$ is
identified with the $G$--Higgs bundle obtained by extending the structure
group of the Higgs $B$--bundle $(F_{B}\, ,\theta_{B})$ using the
inclusion of $B$ in $G$. 

For any holomorphic character \[
\chi\,:\, B\,\longrightarrow\,{\mathbb{C}}^{*}\]
 of $B$, we have 
\begin{equation}
c_{1}(F_{B}\times^{\chi}{\mathbb{C}})\,=\,0\, ,\label{eq:0}\end{equation}
where $F_{B}\times^{\chi}{\mathbb{C}}$ is the holomorphic line bundle
on $Y$ associated to the principal $B$--bundle $F_{B}$ for the
character $\chi$ \cite[p. 20, Theorem 1.1]{BG}.

Let $\mathfrak{g}$ denote the Lie algebra of $G$. We will consider
$\mathfrak{g}$ as a $B$--module using the adjoint action. Since $B$ is solvable, there
is a filtration of $B$--modules
\begin{equation}
0\,=\, V_{0}\,\subset\, V_{1}\,\subset\, V_{2}\,\subset\,\cdots\,\subset\,
V_{d-1}\,\subset\, V_{d}\,=\,\mathfrak{g}\, ,\label{e4}
\end{equation}
where $d\,=\,\dim_{\mathbb{C}}\mathfrak{g}$, such that each successive quotient
$V_i/V_{i-1}$, $1\,\leq\, i\, \leq\, d$, is a $B$--module of dimension one.

Let
$$
W\,:=\, F_{B}\times^{B}\mathfrak{g} \,\longrightarrow\, Y
$$
be the holomorphic vector
bundle on $Y$ associated to the principal $B$--bundle $F_{B}$ for
the above $B$--module $\mathfrak{g}$. We note that this holomorphic vector
bundle $W$ is identified with the adjoint vector bundle \[
(p^{*}F_{G})\times^{G}{\mathfrak{g}}\,=\,\text{ad}(p^{*}F_{G})\,=\, p^{*}\text{ad}(F_{G})\]
for the principal $G$--bundle $p^{*}F_{G}$, because $p^{*}F_{G}$
the the extension of structure group of $F_{B}$ constructed using the inclusion
of $B$ in $G$. So, we write
\begin{equation}
W\,=\,\text{ad}(p^{*}F_{G})\,=\, p^{*}\text{ad}(F_{G})\,.\label{e-1}\end{equation}
 For $0\,\leq\, i\,\leq\, d$, let \[
W_{i}\,:=\, F_{B}\times^{B}V_{i}\,\longrightarrow\, Y\]
 be the holomorphic vector bundle associated to the principal $B$--bundle
$F_{B}$ for the $B$--module $V_{i}$ in \eqref{e4}. The filtration
of $B$--modules in \eqref{e4} produces a filtration of $W$ by holomorphic
vector subbundles
\begin{equation}
0\,=\, W_{0}\,\subset\, W_{1}\,\subset\, W_{2}\,\subset\,\cdots\,\subset\, W_{d-1}
\,\subset\, W_{d}\,=\, W\,=\, p^{*}\text{ad}(F_{G})\, ,\label{e5}
\end{equation}
 where $\text{rank}(W_{i})\,=\, i$ (see \eqref{e-1}).

For any $1\,\leq\, i\,\leq\, d$, the line bundle $W_{i}/W_{i-1}$
on $Y$ coincides with the one associated to the
principal $B$--bundle $F_B$ for the $B$--module $V_i/V_{i-1}$. Therefore, from
\eqref{eq:0} we conclude that
\begin{equation}
c_{1}(W_{i}/W_{i-1})\,=\,0\label{e6}\end{equation}
for all $1\,\leq\, i\,\leq\, d$.

{}From \eqref{e5} and \eqref{e6} we conclude that the vector bundle $p^{*}\text{ad}(F_{G})$
is semistable. This implies that $\text{ad}(F_{G})$ is semistable.
Indeed, if a subsheaf $V'\,\subset\,\text{ad}(F_{G})$ contradicts
the semistability of $\text{ad}(F_{G})$, then the pullback $p^{*}V'$
contradicts the semistability of $p^{*}\text{ad}(F_{G})$. Since $\text{ad}(F_{G})$
is semistable, we conclude that the principal $G$--bundle $F_{G}$
is semistable \cite[p. 214, Proposition 2.10]{AB}. 
\end{proof}

Let $M$ be a compact connected K\"ahler manifold equipped with a K\"ahler form.
A Higgs vector bundle $(E\, ,\theta)$ over $M$ is called
{\it pseudostable} if $E$ admits a filtration by holomorphic subbundles
$$
0\, =\, F_0\,\subset\, F_1 \, \subset\, F_2\,\subset\,
\cdots \, \subset\, F_{n-1}\,\subset\, F_n\, =\, E
$$
such that
\begin{enumerate}
\item{} $\theta(F_i) \, \subset\, F_i\otimes {\Omega}^1_M$
for all $i\, \in\, [1\, , n]$,

\item{} for each integer $i\, \in\, [1\, , n]$, the Higgs vector bundle defined
by the quotient $F_i/F_{i-1}$ equipped with the Higgs field induced by $\theta$
is stable, and

\item{} $\text{degree}(F_{1})/\text{rank}(F_{1})\, =\, \text{degree}(F_{2})/\text{rank}(F_{2})
\, =\, \cdots \,=\,\text{degree}(F_{n})/\text{rank}(F_{n})$, where the degree is defined using
the K\"ahler form on $M$.
\end{enumerate}
A pseudostable Higgs vector bundle is semistable (see \cite{BG}). A holomorphic vector bundle
$E$ on $M$ is called pseudostable if the Higgs vector bundle $(E\, ,0)$ is pseudostable.
A $G$--Higgs bundle $(E_G\, ,\theta)$ on $M$
is called {\it pseudostable} if the adjoint vector
bundle $\text{ad}(E_G)\,=\, E_G\times^G \mathfrak{g}$ equipped with the Higgs field
induced by $\theta$ is pseudostable. A holomorphic principal $G$--bundle $E_G$ on $M$
is called pseudostable if the $G$--Higgs bundle $(E_G\, ,0)$ is pseudostable.

\begin{proposition}
\label{prop2} 
Let $X$ be a compact connected K\"ahler manifold such that
$\pi_{1}(X,x_{0})$ is virtually nilpotent. Let $(F_G\, ,\theta)$ be a
pseudostable $G$--Higgs bundle on $X$ with zero characteristic classes of positive degrees. 
Then the holomorphic principal $G$--bundle $F_G$
is pseudostable.
\end{proposition}

\begin{proof}
The proof of Proposition \ref{prop2} is very similar to the proof of Proposition
\ref{prop1}. Since $(F_G\, ,\theta)$ is pseudostable with vanishing characteristic classes of
positive degrees, it corresponds to a flat $G$--bundle on $X$ \cite[p. 20, Theorem 1.1]{BG}.
Let $(E_G\, ,\nabla)$ be the flat $G$--bundle corresponding to $(F_G\, ,\theta)$. Construct $\rho$
as in \eqref{e2}. Define $\Gamma_0$ as in \eqref{e3}, and consider $p$ as in \eqref{cp}. Now the
proof proceeds exactly as the proof of Proposition \ref{prop1} does. The only point to note is that 
the adjoint vector bundle $\text{ad}(F_G)$, which we get at the end, is pseudostable. 
But this means that $F_G$ is pseudostable (see the above definition).
\end{proof}

\section{Deformation retraction of character varieties}

As before, $X$ is a compact connected K\"ahler manifold such that $\Gamma\,:=\,\pi_{1}(X,x_{0})$
is virtually nilpotent. Since $\Gamma$ is a finitely presented group, and $G$ is an affine algebraic
group, the representation space \[
\widetilde{\mathcal{R}}_{\Gamma}(G)\,:=\,\text{Hom}(\Gamma,\, G)\]
 is an affine algebraic scheme over $\mathbb{C}$. The reductive group
$G$ acts on $\widetilde{\mathcal{R}}_{\Gamma}(G)$ via the conjugation
action of $G$ on itself. The geometric invariant theoretic quotient
\begin{equation}
{\mathcal{R}}_{\Gamma}(G)\,:=\,\widetilde{\mathcal{R}}_{\Gamma}(G)/\!\!/G\label{e8}\end{equation}
is also an affine algebraic scheme over $\mathbb{C}$.

Fix a maximal compact subgroup \[
K\,\subset\, G\,.\]
 Let \begin{equation}
{\mathcal{R}}_{\Gamma}(K)\,:=\,\text{Hom}(\Gamma,\, K)/K\label{e9}\end{equation}
 be the space of all equivalence classes of homomorphisms from $\Gamma=\pi_{1}(X,x_{0})$
to $K$. The inclusion of $K$ in $G$ produces an inclusion \begin{equation}
{\mathcal{R}}_{\Gamma}(K)\,\hookrightarrow\,{\mathcal{R}}_{\Gamma}(G)\, ,\label{e10}\end{equation}
 where ${\mathcal{R}}_{\Gamma}(G)$ and ${\mathcal{R}}_{\Gamma}(K)$
are constructed in \eqref{e8} and \eqref{e9} respectively (see \cite[Proposition 4.5 and Theorem 4.3]{FL2}).

\begin{theorem}\label{thm1}
Let $X$ be a connected smooth complex projective variety with
a virtually nilpotent fundamental group $\Gamma\,=\,\pi_{1}(X,x_{0})$.
Let $G$ be a reductive linear algebraic group.
The character variety ${\rm Hom}(\Gamma,\, G)/\!\!/G$
admits a strong deformation retraction to the subset \[
{\rm Hom}(\Gamma,\, K)/K\,\subset\,{\rm Hom}(\Gamma,\, G)/\!\!/G\,.\]
\end{theorem}

\begin{proof}
We will construct a continuous map
\begin{equation}
\Phi\,:\,{\mathbb{C}}\times{\mathcal{R}}_{\Gamma}(G)\,\longrightarrow\,
{\mathcal{R}}_{\Gamma}(G)\,.\label{e11}
\end{equation}
Take any $\rho\,\in\,{\mathcal{R}}_{\Gamma}(G)$. Choose a homomorphism
$\widetilde{\rho}\,\in\, \widetilde{\mathcal{R}}_{\Gamma}(G)$ (see \eqref{e8})
that projects to $\rho$. Let $(E_{G}\, ,\nabla)$
be the flat principal $G$--bundle on $X$ associated to $\widetilde{\rho}$. Let
$(F_{G}\, ,\theta)$ be the semistable $G$--Higgs bundle on $X$ associated
to the flat principal $G$--bundle $(E_{G}\, ,\nabla)$. From Proposition
\ref{prop1} we know that $F_{G}$ is semistable. Therefore, $(F_{G}\, ,\lambda\cdot\theta)$
is a semistable Higgs $G$--bundle for every $\lambda\,\in\,\mathbb{C}$.
Hence $(F_{G}\, ,\lambda\cdot\theta)$ corresponds to a flat principal
$G$--bundle on $X$. Let $(E^\lambda_{G}\, ,\nabla^\lambda)$ be the flat principal
$G$--bundle on $X$ corresponding to $(F_{G}\, ,\lambda\cdot\theta)$. The map $\Phi$ in \eqref{e11}
sends the point $(\lambda\, ,\rho)\,\in\,{\mathbb{C}}\times{\mathcal{R}}_{\Gamma}(G)$
to the monodromy representation of the flat connection $(E^\lambda_{G}\, ,\nabla^\lambda)$.
The bijection between ${\mathcal{R}}_{\Gamma}(G)$ and the moduli space of semistable
$G$--Higgs bundles on $X$ with vanishing characteristic classes of positive degrees is continuous.
Also, the action of $\mathbb C$ on this moduli space of semistable
$G$--Higgs bundles is continuous. Therefore, $\Phi$ is a continuous map.

Clearly, $\rho\,\longmapsto\,\Phi(1\, ,\rho)$ is the identity map
of ${\mathcal{R}}_{\Gamma}(G)$. We have $\Phi(\lambda\, ,\rho)\,=\,
\rho$ for every $\rho\, \in\, {\mathcal{R}}_{\Gamma}(K)$ and $\lambda\,\in\,
\mathbb C$. Also \[
\rho\,\longmapsto\,\Phi(0\, ,\rho)\]
is a retraction to the subset ${\mathcal{R}}_{\Gamma}(K)$ in \eqref{e10}. 
\end{proof}

\begin{theorem}
\label{thm2} Let $X$ be a compact connected K\"ahler manifold with
a virtually nilpotent fundamental group $\Gamma\,=\,\pi_{1}(X,x_{0})$.
Let $G$ be a reductive linear algebraic group.
Then the character variety ${\rm Hom}(\Gamma,\, G)/\!\!/G$
admits a strong deformation retraction to the subset \[
{\rm Hom}(\Gamma,\, K)/K\,\subset\,{\rm Hom}(\Gamma,\, G)/\!\!/G\,.\]
\end{theorem}

In view of Proposition \ref{prop2}, the proof of Theorem \ref{thm2} is identical
to the proof of Theorem \ref{thm1}.

\begin{remark}
In Theorem \ref{thm1}, the assumption that $\pi_{1}(X,x_{0})$ is virtually nilpotent
is only used in deducing the following: if $(E_G\, ,\theta)$ is a semistable $G$--Higgs
bundle on the complex projective manifold $X$ such that all the characteristic classes of $E_G$ of
positive degree vanish, then the principal $G$--bundle $E_G$ is semistable. Similarly, in
Theorem \ref{thm2}, the assumption that $\pi_{1}(X,x_{0})$ is virtually nilpotent
is only used in deducing the following: if $(E_G\, ,\theta)$ is a pseudostable $G$--Higgs
bundle on the compact connected K\"ahler manifold $X$ such that all the characteristic classes of
$E_G$ of positive degree vanish, then the principal $G$--bundle $E_G$ is pseudostable. It is natural
to ask which other complex projective varieties (or compact connected K\"ahler manifolds) satisfy
this condition on $G$--Higgs bundles.
\end{remark}

\section*{Acknowledgements}

We thank the referee for helpful comments.
The first author is supported by the J. C. Bose fellowship. The
second author is partially supported by FCT (Portugal) through the projects
EXCL/MAT-GEO/0222/2012, PTDC/MAT/120411/2010 and PTDC/MAT-GEO/0675/2012.


\end{document}